\documentclass{article}

\usepackage[utf8]{inputenc}
\usepackage{amsmath,amssymb,amsthm}
\usepackage[dvips]{graphicx}
\usepackage[frenchb]{babel}

\usepackage{verbatim}

\newcommand{\N}{\ensuremath{\mathbb{N}}}

\newcommand{\E}{\ensuremath{\mathbb{E}}}

\renewcommand{\P}{\ensuremath{\mathbb{P}}}

\renewcommand{\epsilon}{\varepsilon}

\renewcommand{\liminf}{\underline{\lim}\quad}

\newcommand{\miniop}[3]{%
\renewcommand{\arraystretch}{0.6}
\begin{array}{c}
{\scriptstyle #1}\\
#2\\
{\scriptstyle #3}
\end{array}
\renewcommand{\arraystretch}{1}}

\newcommand{\1}{1\hspace{-2.7mm}1}
\newcommand{\Var}{\text{Var }}

\newtheorem{theo}{Théorème}
\newtheorem{lemme}{Lemme}
\newtheorem{coro}{Corollaire}

\author{Olivier \textsc{Garet}}
\title{Le processus de Galton-Watson critique s'éteint}
\begin{document}

\maketitle
\begin{abstract}
  Les chaînes de Galton-Watson, ou processus de branchements, font partie de la formation traditionnelle des étudiants en probabilité. Le théorème de base concerne l'étude de la survie du modèle en fonction de la fécondité: sauf cas dégénéré, la survie n'est possible que si le taux de fécondité dépasse $1$.  La preuve classiquement enseignée est essentiellement analytique, reposant sur les fonctions génératrices et des arguments de convexité. On se propose ici, s'inspirant des travaux de Bezuidenhout et Grimmett, de donner une preuve plus conforme à l'intuition probabiliste. 
\end{abstract}

\section{Introduction}
S'inspirant d'un article de Grimmett et Marstrand sur certaines propriétés de la percolation en dimension $d\ge 3$, Bezuidenhout et Grimmett ont démontré dans un article célèbre que le processus de contact s'éteint au point critique.
Leur technique de preuve a souvent été utilisée pour étudier la croissance des modèles de population.

 Le présent texte se veut une introduction à leurs idées, dans le cadre de l'exemple simple du processus de Galton-Watson. On montrera ici qu'un processus de Galton-Waltson non dégénéré ne peut s'éteindre que si sa fertilité dépasse strictement~1. La preuve classiquement enseignée -- voir par exemple Benaïm--El Karoui~\cite{BEK} -- est essentiellement analytique. Elle repose sur les fonctions génératrices et des arguments de convexité.
La preuve présentée ici ne fait évidemment pas usage des fonctions génératrices et permet ainsi de retrouver une preuve plus conforme à l'intuition probabiliste.

\section{Chaînes de Galton-Walson: définition et propriétés élémentaires}

Soit $\nu,\mu$ deux lois sur $\N$.
$\nu$ est appelée loi de reproduction et $\mu$ 
est la loi de la taille de la population initiale. 

On appelle  
chaîne de Galton-Waltson de loi initiale $\mu$ et de loi de reproduction $\nu$
la chaîne de Markov de loi initiale $\mu$ et de matrice de transition
\begin{equation*}
p_{i,j}=
\begin{cases}
\nu^{*i}(j)\text{ si }i\ne 0\\
\delta_0(j)\text{ si }i=0
\end{cases}
\end{equation*}

On peut fabriquer une telle chaîne comme suit:
Soient $(X_i^n)_{i,j\ge 1}$ des variables aléatoires de loi $\nu$ et $Y_0$
une variable aléatoire de loi $\mu$ indépendante des $(X_i^n)_{i,j\ge 1}$.
On définit par récurrence la suite $(Y_n)_{n\ge 1}$ par 
$$\forall n\ge 0\quad Y_{n+1}=\sum_{1\le i\le Y_n}X_i^n.$$
Alors $(Y_n)_{n\ge 0}$ est une chaîne de Galton-Watson de loi initiale $\mu$ et de loi de reproduction $\nu$.

Si je note $\mathcal{F}_n=\sigma(X_i^k,i\ge 1,k\le n)$ et $m=\int_{\N} x\ d\nu(x)$, on a classiquement
\begin{equation}
\label{puissance}
\E[Y_{n+1}]=m\E[Y_n]\text{ et }\E[Y_n]=m^n\E[Y_0]
\end{equation}
On pose $\tau=\inf\{n\ge 0; Y_n=0\}$.

\begin{theo}
\label{mort}
Si $m<1$, $\P(\tau>n)=O(m^n)$. En particulier $\P(\tau<+\infty)=1$.
\end{theo}
\begin{proof}
Avec~\eqref{puissance}, on a $\P(\tau>n)\le\P(Y_n\ge 1)\le \E[Y_n]=m^n\E[Y_0]$.
\end{proof}

\begin{theo}
\label{galtonindep}
Si $(X_n)_{n\ge 0}$ et  $(Y_n)_{n\ge 0}$ sont deux chaînes de Galton-Waltson indépendantes de même  loi de reproduction $\nu$, alors 
 $(X_n+Y_n)_{n\ge 0}$ est encore une chaîne de Galton-Waltson de loi de reproduction $\nu$.
\end{theo}
\begin{proof}
Comme $(X_n)_{n\ge 0}$ et  $(Y_n)_{n\ge 0}$ sont des chaînes de Markov indépendantes, $((X_n,Y_n))_{n\ge 0}$ est une chaîne de Markov, de matrice de passage
$$p_{(x,a),(y,b)}=\nu^{* x}(a)\nu^{* y}(b).$$ Notons 
$\P^{(x,y)}$ les lois des chaînes de Markov canoniquement associées.
Il s'agit maintenant de montrer que si l'on pose $f(x,y)=x+y$, alors $(f(X_n,Y_n))_{n\ge 0}$ est une chaîne de Markov.
Pour cela, il suffit de montrer que si $x+y=r$, alors
$\P^{(x,y)}(f(X_1,Y_1)=p)$ ne dépend que $r$ et $p$.
Or, sous  $\P^{(x,y)}$, $X_1$ et $Y_1$ sont deux variables aléatoires indépendantes de lois respectives $\nu^{* x}$ et $\nu^{* y}$, donc la loi de
 $f(X_1,Y_1)$ est  $\nu^{* x}* \nu^{* y}=\nu^{* (x+y)}=\nu^{* r}$.
Ainsi, $\P^{(x,y)}(f(X_1,Y_1)=p)=\nu^{* r}(\{p\})$ et 
$(X_n+Y_n)_{n\ge 0}$ est bien une chaîne de Galton-Waltson de loi de reproduction $\nu$. Comme la loi initiale est $\P_{X_0+Y_0}=\P_{X_0}*\P_{Y_0}=\mu_1*\mu_2$, on a le résultat voulu.
\end{proof}

Dans la suite on notera $\P^i$ une probabilité sous laquelle une suite
$(Y_n)_{n\ge 0}$ est une chaîne de Galton-Watson de loi initiale $\delta_{i}$ et de loi de reproduction $\nu$.

\begin{coro}
Soient $n,p\ge 0$. On a
\begin{itemize}
\item Pour tout $n\ge 0$, $\P^n(\tau<+\infty)=\P^{1}(\tau<+\infty)^n$
\item Pour $n,p\ge 0$ $\P^n(\tau<+\infty|\mathcal{F}_p)=\P^{1}(\tau<+\infty)^{Y_p}$.
\item Pour $n,p\ge 1$, on a $\P^n(\tau=+\infty)>0 \iff \P^p(\tau=+\infty)>0.$
\end{itemize}
\end{coro}

\begin{proof}
Grâce au théorème~\ref{galtonindep} on a $$\P^{n+1}(\tau<+\infty)=\P^{n}(\tau<+\infty)\P^{1}(\tau<+\infty),$$ d'où par récurrence $\P^{n}(\tau<+\infty)=\P^{1}(\tau<+\infty)^n$. Cela donne le premier point. Avec la propriété de Markov, on obtient alors le deuxième point. Le dernier point est évident.
\end{proof}

\begin{coro}
\label{souschaine}
Soit $T\ge 1$. $(Y_{Tn})_{n\ge 0}$ est une chaîne de Galton-Watson de loi de reproduction $\P^1_{Y_T}$.
\end{coro}
\begin{proof}
Comme $(Y_n)$ est une chaîne de Markov, il est bien  connu que $(Y_{Tn})_{n\ge 0}$ est une chaîne de Markov. Il reste à évaluer les probabilités de transition.

Soit $k\ge 1$.
En appliquant $k-1$ fois le théorème~\ref{galtonindep},on  voit que
si  $(Y^1_t)_{t\ge 0}$,$(Y^2_t)_{t\ge 0},$\dots$ (Y^k_t)_{t\ge 0}$ sont des processus de Galton-Watson indépendants de lois initiales respectives $\delta_1$ et de loi de reproduction~$\nu$,
alors $(Y^1_t+\dots Y^k_t)_{t\ge 0}$ est un processus de  Galton-Watson de loi initiale $\delta_{k}$ et de loi de reproduction~$\nu$.
En conséquence
$$\P^k(Y_T=p)=\P(Y^1_T+\dots Y^k_T=p)=\P_{Y^1_T}^{*k}(p),$$
ce qui est le résultat voulu.
\end{proof}

\section{Une preuve probabiliste}

\subsection{Survie dans la zone surcritique}

\begin{theo}
\label{surcritique}
Si $m>1$, alors $\P^{1}(\tau=+\infty)>0$.
\end{theo}
\begin{proof}
Soit $a$ avec $1<a<m$. On a
$$\miniop{}{\lim}{M\to +\infty} \int x\wedge M \ d\nu= \int  x \ d\nu=m,$$ donc il existe $M$ tel que  $\int x\wedge M \ d\nu>a$.  Pour $k\ge n$, on a

\begin{align*}
\P^k(Y_1<na)&=\P(X_1+\dots X_k<na)\\\le &\P(X_1\wedge M+\dots X_n\wedge M<na)\\& =\P(n\E[X_1\wedge M]-(X_1\wedge M+\dots X_n\wedge M))> (\E[X_1\wedge M]-a)n)\\& \le\frac{\Var X_1\wedge M}{(\E[X_1\wedge M]-a)n},
\end{align*}
avec l'inégalité de Tchebitchef. 
Posons $\phi(k,x)=\P^k(Y_1<x)$.
Soit $n>c=\frac{\Var (X_1\wedge M)}{\E[X_1\wedge M]-a}$
Avec la propriété de Markov, on a pour
 $A\in \mathcal{F}_t$ avec $A\subset\{Y_t\ge n\}$:
\begin{align*}
\P(A\cap \{Y_{t+1}<an\})&=
\E[\1_A  \1_{Y_{t+1}<an\}}]\\
&=\E[\1_A\E[\1_{Y_{t+1}<an\}}|\mathcal{F}_t]]\\
&=\E[\1_A \phi(Y_t,an)]\\
&\le \E [\1_A c/n]=c/n\P(A),
\end{align*}
d'où $\P(Y_{t+1}\ge an| A)\ge 1-\frac{c}n$.\\
On en déduit par récurrence que pour $A_t=\miniop{t}{\cap}{i=1}\{Y_{t}\ge na^t\}$, on a
$$\P^n(A_t)\ge\miniop{t-1}{\prod}{i=0}\left(1-\frac{c}{na^i}\right),$$
d'où
$\P^n(\tau=+\infty)\ge\P^n(\forall t\ge 0\quad Y_{t}\ge na^t)\ge\prod_{i=0}^{+\infty} \left(1-\frac{c}{na^i}\right)>0.$

\end{proof}

\subsection{La survie est une propriété locale}

\begin{theo}
\label{equivalence}
Soit $(Y_n)_{n\ge 0}$ une chaîne de Galton-Watson de loi de reproduction~$\nu$.
On suppose que $\nu(0)>0$. Alors on a équivalence entre
\begin{itemize}
\item $\exists N,T\ge 1\quad \P^N(Y_T\ge 2N)>\frac12$.
\item $\P^1(\tau=+\infty)>0$.
\end{itemize}
\end{theo}

Avant de faire la preuve, donnons un aperçu des idées principales:
\begin{itemize}
\item Dans le sens direct, l'idée est de comparer la chaîne avec  un processus de Galton-Watson surcritique, puis de conclure à l'aide du théorème~\ref{surcritique}.
\item Le sens réciproque est ici assez simple, puisqu'il s'agit essentiellement de montrer que le nombre de particules explose dès qu'il y a survie.
Toutefois, il faudra avoir à l'esprit que si l'évènement local est plus compliqué, cette partie-là sera en réalité la plus difficile.
\end{itemize}

\begin{lemme}
 Si il existe $a$ et $n$ positifs tels que $a\P^N(Y_1\ge aN)>1$, alors $\P^1(\tau=+\infty)>0$.
\end{lemme}
\begin{proof}
Soient $X_i^n$ i.i.d. de loi $\nu$.
On pose $M_0=1$, $Y_0=N$, puis 
$$\forall n\ge 0\quad Y_{n+1}=\sum_{1\le i\le Y_n}X_i^n\text{ et }M_{n+1}=\sum_{i=1}^{M_n} aB_i^n,$$
avec $B_i^n=\1_{\{X^n_{(i-1)N+1}+\dots X^n_{iN}\ge aN\}}$.
On montre par récurrence que $Y_n\ge NM_n$. En effet, si $Y_n\ge NM_n$, on a
$$Y_{n+1}=\sum_{1\le i\le Y_n}X_i^n\ge \sum_{1\le i\le NM_n}X_i^n=\sum_{i=1}^{M_n}(X^n_{(i-1)N+1}+\dots X^n_{iN})\ge \sum_{i=1}^{M_n} aNB_i^n=NM_{n+1}.$$
$(M_n)$ est une chaîne de Galton-Walson de fertilité $\E[aB_i^n]=a\P^N(Y_1\ge aN)>1$, donc qui peut survivre d'après le théorème~\ref{surcritique}. Par comparaison, $(Y_n)$ peut survivre.
\end{proof}
On remarquera que la preuve du lemme repose sur un argument de couplage: on fait vivre sur un même espace $(Y_n)_{n\ge 0}$ et un processus de Galton-Watson de loi de reproduction $(1-q)\delta_0+q\delta_{a}$, avec $q=\P^N(Y_1\ge aN)$.

\begin{proof}[Preuve du théorème]
D'après le corollaire~\ref{souschaine}, $(Y_{nT})_{n\ge 0}$ est une chaîne de Galton-Watson. On peut donc lui appliquer le lemme avec $a=2$: $(Y_{nT})_{n\ge 0}$ peut survivre, donc  $(Y_{n})_{n\ge 0}$ peut survivre.

Réciproquement, supposons que $\P^1(\tau=+\infty)>0$, soit $\P^1(\tau<+\infty)<1$.\\
Comme $\P^N(\tau<+\infty)=\P^1(\tau<+\infty)^N$, donc on peut choisir
$N$ tel que $\P^N(\tau<+\infty)<1/2$.

On a vu que $\P^N(\tau<+\infty|\mathcal{F}_t)=\P^1(\tau<+\infty)^{Y_t}$.\\
Comme $\P^1(\tau<+\infty)\ge \P^1(Y_1=0)=\nu(0)>0$, on peut écrire
$$Y_t=\frac{\log \P^N(\tau<+\infty|\mathcal{F}_t)}{\log \P^1(\tau<+\infty)}.$$
Or le théorème de convergence des martingales dit que
$$\E^N[\1_{\{\tau<+\infty\}}|\mathcal{F}_t]=\P^N(\tau<+\infty|\mathcal{F}_t)\to \1_{\{\tau<+\infty\}}\quad\P^N\text{ p.s.}$$ lorsque $t$ tend vers l'infini. \\
En particulier, sur l'événement $\{\tau<+\infty\}$, 
$\P^N(\tau<+\infty|\mathcal{F}_t)$ tend presque sûrement vers $0$ et
$Y_t$ tend presque sûrement vers l'infini. On a donc presque sûrement l'inégalité
$$\1_{\{\tau=+\infty\}}\le \miniop{}{\liminf}{t\to +\infty}\1_{\{Y_t\ge 2N\}}.$$
Avec le lemme de Fatou, il vient
$$\P^N(\tau=+\infty)=\E^N(\1_{\{\tau=+\infty\}})\le \miniop{}{\liminf}{t\to +\infty}\E^N[\1_{\{Y_t\ge 2N\}}]= \miniop{}{\liminf}{t\to +\infty}\P^N(Y_t\ge 2N)$$
Comme $\P^N(\tau=+\infty)>1/2$, il existe $T$ tel que $\P^N(Y_T\ge 2N)>1/2$.
\end{proof}

\subsection{\'Etude du cas critique}

\begin{theo}
Si $\nu(0)>0$ et $m=1$, alors $\P^1(\tau=+\infty)=0$.
\end{theo}

\begin{proof}[Preuve 1]
Il suffit de noter que quels que soient $N,T\ge 1$, on a $$\P^N(Y_T\ge 2N)\le\frac{\E^N(Y_T)}{2N}=\frac{N}{2N}=\frac12$$
et d'appliquer la contraposée du théorème~\ref{equivalence}.
\end{proof}
On va voir une deuxième preuve, un peu plus longue, mais aussi plus robuste.
Cette méthode a notamment été utilisée dans Garet-Marchand~\cite{GM-BRW} pour l'étude des marches aléatoires branchantes en milieu aléatoire.
\begin{proof}[Preuve 2]
Supposons par l'absurde que l'on a non seulement $\nu(0)>0$ et $m=1$ mais aussi $\P^1(\tau=+\infty)>0$.

D'après le théorème~\ref{equivalence} (sens réciproque), on peut choisir $n$ et $T$ tels que  $\P^N(Y_T\ge 2N)>\frac12$.

L'idée est alors de coupler la chaîne avec une chaîne sous-critique.
Soient $(X_i^n)_{i,j\ge 1}$ des variables aléatoires de loi $\nu$, $(B_i^n)_{i,j\ge 1}$ une suite de variables de Bernoulli de paramètre $p$ indépendantes  des $(X_i^t)$. On pose $Y_0=N$, $Y^p_0=N$, puis 

$$\forall n\ge 0\quad Y_{n+1}=\sum_{1\le i\le Y_n}X_i^n\text{ et }Y^p_{n+1}=\sum_{1\le i\le Y^p_n}B_i^n X_i^n.$$

Par continuité séquencielle croissante,
$$\miniop{}{\lim}{M\to +\infty}\P^N( \max(Y_i,0\le i\le T)\le M,Y_T\ge 2N)=\P^N( Y_T\ge 2N)>1/2,$$
donc il existe $M$ tel que $\P( \max(Y_i,0\le i\le T)\le M,Y_T\ge 2N)>1/2$.
On a alors
\begin{align*}
\P(Y^p_T\ge 2N)&\ge \P(Y_T\ge 2N,\forall i\le T\quad Y^p_i=Y_i)\\
& \ge \P\left( \begin{array}{l}\max(Y_i,0\le i\le T)\le M,Y_T\ge 2N,\\\forall (t,i)\in\{0,\dots, T-1\}\times\{1,\dots,M\} \quad B_i^t=1\end{array}\right)\\&=\P( \max(Y_i,0\le i\le T)\le M,Y_T\ge 2N)p^{TM}
\end{align*}
Si on prend $p<1$ suffisamment grand, on  a $$\P( \max(Y_i,0\le i\le T)\le M,Y_T\ge 2N)p^{TM}>1/2,$$ donc $\P(Y^p_T\ge 2N)>1/2$.
Or $(Y^p_t)$ est un processus de Galton-Watson de loi de reproduction $B_1^1X_1^1$ et de loi initiale $\delta_N$, donc d'après le théorème~\ref{equivalence} (sens direct), ce processus de Galton-Watson peut survivre.
Mais $$\E[B_1^1X_1^1]=\E[B_1^1]\E[X_1^1]=pm=p<1,$$ donc d'après le théorème~\ref{mort}, le processus ne peut survivre. Contradiction.

\end{proof}

\def\refname{Références}
\bibliographystyle{plain}

\begin{thebibliography}{}

\end{thebibliography}


\begin{thebibliography}{1}

\bibitem{BEK}
Michel Benaïm and Nicole El~Karoui.
\newblock {\em Promenade aléatoire: Chaines de Markov et simulations;
  martingales et stratégies}.
\newblock Ecole Polytechnique, 2004.

\bibitem{MR1071804}
Carol Bezuidenhout and Geoffrey Grimmett.
\newblock The critical contact process dies out.
\newblock {\em Ann. Probab.}, 18(4):1462--1482, 1990.

\bibitem{GM-BRW}
Olivier Garet and R\'egine Marchand.
\newblock The critical branching random walk in a random environment dies out.
\newblock {\em Electron. Comm. Probab.}, 18(9):1--15 (electronic), 2013.

\bibitem{Grimmett-Marstrand}
G.~R. Grimmett and J.~M. Marstrand.
\newblock The supercritical phase of percolation is well behaved.
\newblock {\em Proc. Roy. Soc. London Ser. A}, 430(1879):439--457, 1990.

\end{thebibliography}

\end{document}